\documentclass[12pt,a4paper]{article}

 \usepackage{mathptmx}       
 \usepackage{helvet}         
 \usepackage{courier}        

\usepackage{graphicx}        

\usepackage[a4paper, hmargin={2.5cm,2.5cm},vmargin={3.3cm,3.3cm}]{geometry}
\usepackage{amsmath, amssymb, amsfonts, amsbsy, latexsym, color} 
\usepackage{amsthm}
\usepackage{multirow}
\usepackage{paralist}
\usepackage[T1]{fontenc}
\usepackage[latin1]{inputenc}
\usepackage[english]{babel}

\usepackage[mathscr,psamsfonts]{euscript}
\usepackage{cite}
\usepackage{url}
\usepackage[caption=false]{subfig}
\usepackage{caption}
\usepackage{titletoc}
\usepackage[sf,bf,compact,pagestyles,small]{titlesec}
\titlespacing*{\section}{0pt}{14pt}{4pt}
\titlespacing*{\subsection}{0pt}{8pt}{3pt}

\def\maketimestamp{\count255=\time
\divide\count255 by 60\relax
\edef\thetime{\the\count255:}%
\multiply\count255 by-60\relax
\advance\count255 by\time
\edef\thetime{\thetime\ifnum\count255<10 0\fi\the\count255}
\edef\thedate{\number\day-\ifcase\month\or Jan\or Feb\or Mar\or
             Apr\or May\or Jun\or Jul\or Aug\or Sep\or Oct\or
             Nov\or Dec\fi-\number\year}
\def\timstamp{\hbox to\hsize{\tt\hfil\thedate\hfil\thetime\hfil}}}
\maketimestamp

\usepackage{fancyhdr,lastpage}
\fancypagestyle{plain}{%
\fancyhf{} 
 \fancyfoot[R]{\scriptsize \tt date/time:\,\thedate/\thetime}
 \fancyfoot[C]{\scriptsize  Page \thepage{} of \pageref{LastPage}}
 \fancyfoot[L]{} 
}

\numberwithin{equation}{section}  
\allowdisplaybreaks[4]

\newtheorem{theorem}{Theorem}[section]

\newtheorem{lemma}[theorem]{Lemma}
\newtheorem{proposition}[theorem]{Proposition}
\newtheorem{corollary}[theorem]{Corollary}
\theoremstyle{definition}
\newtheorem{definition}{Definition}[section]
\newtheorem{example}{Example}
\theoremstyle{remark}
\newtheorem{remark}{Remark}

\newtheoremstyle{question}%
{3pt}%
{3pt}%
{}%
{}%
{\hspace{-2em}$\longrightarrow$ \scshape}%
{:}%
{.5em}%
{}%
\theoremstyle{question}

\newcommand{\ip}[2]{\langle#1,#2\rangle}

\DeclareMathOperator{\Span}{span} %
\DeclareMathOperator{\supp}{supp} %

\DeclareMathOperator{\tr}{tr} 

\DeclareMathOperator{\sparky}{spark}
\newcommand{\spark}[2][]{\sparky_{#1}{(#2)}}

\newcommand{\id}[1][]{I_{#1}} 

\newcommand{\card}[1]{\abs{#1}} 

\newcommand*{\numbersys}[1]{\ensuremath{\mathbb{#1}}}
\newcommand*{\C}{\numbersys{C}}
\newcommand*{\R}{\numbersys{R}}

\newcommand*{\N}{\numbersys{N}}
\newcommand*{\scalarfield}[1]{\ensuremath{\mathbb{#1}}}
\newcommand*{\K}{\scalarfield{K}}
\newcommand{\itvoo}[2]{\ensuremath{\left({#1},{#2}\right)}} %
\newcommand{\itvco}[2]{\ensuremath{\left[{#1},{#2}\right)}} %
\newcommand{\abs}[1]{\ensuremath{\left\lvert#1\right\rvert}}

\newcommand{\norm}[2][]{\ensuremath{\left\lVert#2\right\rVert_{#1}}}

\newcommand{\normsmall}[2][]{\ensuremath{\lVert#2\rVert_{#1}}}

\newcommand{\innerprod}[3][]{\ensuremath{\left\langle #2,#3\right\rangle_{\! #1}}}

\newcommand{\innerprods}[2]{\ensuremath{\langle #1,#2\rangle}}
\newcommand{\set}[1]{\ensuremath{\left\lbrace{#1}\right\rbrace}}
\newcommand{\seq}[1]{\ensuremath{\left({#1}\right)}}
\newcommand{\seqsmall}[1]{\ensuremath{({#1})}}
\newcommand{\setprop}[2]{\ensuremath{\left\lbrace{#1} : {#2}\right\rbrace}}

\newcommand{\setsmall}[1]{\ensuremath{\lbrace{#1}\rbrace}}

\newcommand{\setpropsmall}[2]{\ensuremath{\lbrace{#1} : {#2}\rbrace}}

\newcommand{\floor}[1]{\left\lfloor #1 \right\rfloor}

\newcommand{\STF}{\ensuremath{\mathbf{STF}}}

\makeatletter
\renewcommand*\env@matrix[1][c]{\hskip -\arraycolsep
  \let\@ifnextchar\new@ifnextchar
  \array{*\c@MaxMatrixCols #1}}
\makeatother
\newlength{\bracewidth}
\newcommand{\myunderbrace}[2]{\settowidth{\bracewidth}{$#1$}#1\hspace*{-1\bracewidth}\smash{\underbrace{\makebox{\phantom{$#1$}}}_{#2}}}

\usepackage{xspace}
\newcommand{\ie}{i.e.,\xspace} 
\newcommand{\eg}{e.g.,\xspace} 

\usepackage[hyperindex,colorlinks,citecolor=blue]{hyperref} 
\hypersetup{
pdfview={FitH},
pdfstartview={FitH},
pdfauthor = {Felix Krahmer, Gitta Kutyniok, Jakob Lemvig},
pdftitle = {Sparsity and spectral properties of dual frames},
pdfsubject = {frame theory},
pdfkeywords = {dual frames, sparsity, spectrum}
pdfcreator = {LaTeX with hyperref package},
pdfproducer = {dvips + ps2pdf}}

\makeatletter
\def\blfootnote{\xdef\@thefnmark{}\@footnotetext} 
\def\subjclass{\xdef\@thefnmark{}\@footnotetext}
\long\def\symbolfootnote[#1]#2{\begingroup%
\def\thefootnote{\fnsymbol{footnote}}\footnote[#1]{#2}\endgroup} 
\if@titlepage
  \renewenvironment{abstract}{%
      \titlepage
      \null\vfil
      \@beginparpenalty\@lowpenalty
      \begin{center}%
        \bfseries \abstractname
        \@endparpenalty\@M
      \end{center}}%
     {\par\vfil\null\endtitlepage}
\else
  \renewenvironment{abstract}{%
      \if@twocolumn
        \section*{\abstractname}%
      \else
        \small
        \list{}{%
          \settowidth{\labelwidth}{\textbf{\abstractname:}}
          \setlength{\leftmargin}{50pt}
          \setlength{\rightmargin}{50pt}
          \setlength{\itemindent}{\labelwidth}
          \addtolength{\itemindent}{\labelsep}
        }
        \item[\textbf{\abstractname:}]

      \fi}
      {\if@twocolumn\else\endlist\fi}
\fi
\makeatother

\begin{document}


\title{Sparsity and spectral properties of dual frames} 





\author{Felix Krahmer\footnote{Georg-August-Universit\"at G\"ottingen, Institut f\"ur Numerische und Angewandte Mathematik, 37083 G\"ottingen, Germany, E-mail \url{f.krahmer@math.uni-goettingen.de}}\phantom{$\ast$}, 
Gitta Kutyniok\footnote{%
Technische Universit\"at Berlin, Institut f\"ur Mathematik,  10623 Berlin, Germany, E-mail: \url{kutyniok@math.tu-berlin.de}}\phantom{$\dagger$},
Jakob Lemvig\footnote{Technical University of Denmark, Department of Mathematics, Matematiktorvet 303, 2800 Kgs. Lyngby, Denmark,
    E-mail: \url{J.Lemvig@mat.dtu.dk}}}

\date{\today}

\maketitle 
\blfootnote{$2010$ {\it Mathematics Subject Classification.} Primary 42C15; Secondary 15A09, 65F15, 65F20, 65F50}

\thispagestyle{plain}
\begin{abstract} We study sparsity and spectral properties of dual frames of a given finite frame. We show that any finite frame has a dual with no more than $n^2$ non-vanishing entries, where $n$ denotes the ambient dimension, and that for most frames no sparser dual is possible. Moreover, we derive an expression for the exact sparsity level of the sparsest dual for any given finite frame using a generalized notion of spark. We then study the spectral properties of dual frames in terms of singular values of the synthesis operator. We provide a complete characterization for which spectral patterns of dual frames are possible for a fixed frame. For many cases, we provide simple explicit constructions for dual frames with a given spectrum, in particular, if the constraint on the dual is that it be tight.   
\end{abstract}



\section{Introduction}
\label{sec:introduction}
Redundant representations of finite dimensional data using finite frames have received considerable attention over the last years (see \cite{CaKu12} for an extensive treatment including many references concerning various aspects). Finite frames are overcomplete systems in a finite dimensional Hilbert space. Similar to a basis expansion, frames can be used both to represent a signal in terms of its inner products with the frame vectors, the \emph{frame coefficients}, and to reconstruct the signal from the frame coefficients. As for bases, the frames used for representation and reconstruction depend on each other, but are not, in general, the same. In contrast to basis expansions, however, even once one of the two frames is fixed, the other, the \emph{ dual frame}, is not unique. Hence the quality of a method for redundant representation will depend on both the frame and the dual frame. There is, however, a standard choice for the dual frame, the so-called \emph{canonical dual}, corresponding to the Moore-Penrose pseudo-inverse, whose optimality in many respects follows directly from its geometric interpretation. This is why most previous work concerning the design of finite frames has focused primarily on the design of single frames rather than dual frame pairs, as it would be necessary to take into account all degrees of freedom in the problem. In particular, design according to spectral properties \cite{CFMPS11} and sparsity \cite{optimally2011} has recently received attention. 

In this paper, we will take a different viewpoint; we study the set of dual frames for a given fixed finite frame. The guiding criteria, however, will again be sparsity and spectral properties. This is motivated by work in the recent papers \cite{krahmer_Sparse_duals, shidong_l1-min-duals}, which ask the following question:
 If the
frame $\{\phi_i\}_{i=1}^m$ is given by the application at hand, \eg by
the way of measuring the data, which dual for
the synthesis process is the best to choose? The precise answer to
this question is, of course, dependent on the application, but
universal desirable properties of the dual can, nonetheless, be
recognized. Among such desirable properties are fast and stable
reconstruction. Measures for the \emph{stability} of the reconstructions resulting from a dual frame are extensively discussed in \cite{krahmer_Sparse_duals}. It is shown that the computational properties of the dual frames are directly linked to spectral properties of the dual frame matrix. In particular, the Frobenius norm and the spectral norm play an important role in this context. The analysis in \cite{krahmer_Sparse_duals} focuses on stability estimates when the frame coefficient vector is corrupted by random or uniform noise. The papers \cite{Erasure1, Erasure2, ErasureP} discuss an alternative scenario of corruptions through erasures. The fact that the canonical dual is shown to be optimal in many cases also suggests connections to spectral properties, but we will not explore this connection further in this paper, leaving it to potential further work.

A \emph{fast} reconstruction process is closely linked to fast
matrix-vector multiplication properties of the dual frame matrix. Matrices with
these properties
can be constructed in various ways. However, many such constructions yield matrices with a 
specific structure, often linked to the fast Fourier transform, which can be counteracting
the constraint that the matrix should be a dual to the given frame.
The criterion of sparsity,  while also connected to fast matrix-vector multiplication properties, 
does not lead to such problems, but rather provides a way to ensure both duality to a given frame and a fast matrix-vector multiply. The first paper on the topic \cite{shidong_l1-min-duals} constructs sparse duals of
Gabor frames using $\ell_1$-minimization, and 
\cite{krahmer_Sparse_duals} also uses sparsity as a measure for computational efficiency.

The focus of \cite{krahmer_Sparse_duals} has been on deriving criteria for comparing computational properties of dual frames and empirically balancing them to find a good dual frame. This paper is meant to be complementary to \cite{krahmer_Sparse_duals} as we analytically study the set of all dual frames of a given frame with respect to sparsity and spectral properties. We ask how sparse a dual frame can be and what we can say about the spectrum. 

Our paper is organized as follows. In Section~\ref{sec:setup} we present more formal definitions of frames and their duals as well as a review of results in frame theory, an overview of the notation we will be using, and some basic observations.
 In Section~\ref{sec:sparsity-duals} we then consider the sparsity properties of the set of all dual frames.  Finally, Section~\ref{sec:spectrum-dual} focuses on the spectral picture of the set of all dual frames.

\section{Setup and basic observations\label{sec:setup}}
\subsection{Frames and their duals}
Throughout this paper, we are considering signals $x$ in a $n$-dimensional Hilbert space $\mathcal{H}_n$. We will often identify
$\mathcal{H}_n$ with $\K^n$, where $\K$ denotes either $\R$ or $\C$, using a fixed orthonormal basis of
$\mathcal{H}_n$. We write this canonical identification as $\mathcal{H}_n \cong \K^n$. Thus, $\K^n$ or $\mathcal{H}_n$
will be the signal space and $\K^m$ the coefficient space. Then the precise definitions of frames and their dual frames are the following.

\begin{definition}\label{def:framedef}
  A collection of vectors $\Phi=\seq{\phi_i}_{i=1}^m\subset\mathcal{H}_n$ is called a \emph{frame} for $\mathcal{H}_n$ if there are two
  constants $0<A\leq B$ such that
\begin{equation}\label{eq:frameineq}
A\norm[2]{x}^2 \leq \sum_{i=1}^m \abs{\innerprod{x}{\phi_i}}^2 \leq B\norm[2]{x}^2,\quad\text{for all $x\in\mathcal{H}_n$.}
\end{equation}
If the frame bounds $A$ and $B$ are equal, the frame $\seq{\phi_i}_{i=1}^m$ is called \emph{tight frame} for $\mathcal{H}_n$. 
\end{definition}

Let $\Phi=\seq{\phi_i}_{i=1}^m$ be a frame for $\mathcal{H}_n$.  Since we identify $\mathcal{H}_n \cong \K^n$, we also
write $\phi_i$ for the coefficient (column) vector of $\phi_i$ and $\Phi=[\phi_i]_{i=1}^m \in \K^{n \times m}$ as a $n
\times m $ matrix, and we say that $\Phi$ is a frame for $\K^n$. We write $\phi^j$ for the $j$th row of $\Phi$, and
$\Phi^{(k)}=[\phi^j]_{j=1,j\neq k}^m \in \K^{(n-1) \times m}$ as a $(n-1) \times m$ submatrix of $\Phi$ with the $k$th
row deleted; the $j$th column of $\Phi^{(k)}$ is denoted by $\phi_j^{(k)}\in \K^{n-1}$. We also use the notation $\Phi=[\phi_{i,j}]$ or $[\Phi]_{i,j}=\phi_{i,j}$. For an index set $I \subset
\set{1,\dots,m}$, we denote the restriction of $\Phi$ to these column vectors by $\Phi_I=[\phi_i]_{i \in I}$. 

%
\begin{definition}\label{def:dualframedef}
 Given a frame $\Phi$, another frame $\Psi=\seq{\psi_i}_{i=1}^m\subset\mathcal{H}_n$ is said to be a \emph{dual frame} of $\Phi$ if the following
reproducing formula holds:
\begin{equation}
 x = \sum_{i=1}^m \innerprod{x}{\phi_i}\psi_i \quad \text{for all $x \in \mathcal{H}_n$}.\label{eq:-dual-frames-formula}
\end{equation}
\end{definition}
In matrix notation this definition reads 
\begin{equation}
 \Psi \Phi^\ast = I_n \label{eq:dual-eq-2}
\end{equation}
or, equivalently, $\Phi \Psi^\ast = I_n$, where $I_n$ is the $n \times n$ identity
matrix. So the set of all duals of $\Phi$ is the set of all left-inverses $\Psi$ to $\Phi^\ast$ (or the adjoints
of all right-inverses to $\Phi$).

If $m>n$, then the frame $\seq{\phi_i}_{i=1}^m$ is \emph{redundant}, i.e., it consists of more vectors than necessary for the spanning property. For these frames 
there exist infinitely many dual frames.

Given a collection of $m$ vectors $\Phi=\seq{\phi_i}_{i=1}^m$ in $\mathcal{H}_n$, the \emph{synthesis operator} is defined as
the mapping
\begin{equation*}
   F: \K^m \rightarrow\mathcal{H}_n , \quad (c_i)_{i=1}^m\mapsto \sum_{i=1}^m c_i \phi_i.
 \end{equation*}
The adjoint of the synthesis operator is called the \emph{analysis operator} and is given by
 \begin{equation*}
 F^\ast: \mathcal{H}_n \rightarrow \K^m, \quad x\mapsto \big(\ip{x}{\phi_i}\big)_{i=1}^m.
 \end{equation*} 
In applications the analysis operation typically represents the way data is measured, whereas the synthesis procedure represents the reconstruction of the signal from the data.

 The identification $\mathcal{H}_n \cong \K^n$ translates to identifying the synthesis mapping $F$ with the matrix $\Phi$ and $F^\ast$ with
 $\Phi^\ast$. The \emph{frame operator} $S:\mathcal{H}_n \to \mathcal{H}_n$ is defined as $S=FF^\ast$, that is,
\begin{equation}
  S(x) =\sum_{i=1}^m \ip{x}{\phi_i}\phi_i ; \label{eq:frameop}
\end{equation}
in matrix form the definition reads $S= \Phi \Phi^{*}$. It is easy to see that $\Phi$ is a frame if and
only if the frame operator defined by \eqref{eq:frameop} is positive definite. In this case, the following
reconstruction formula holds
\begin{equation}\label{eq:reconstruction_for_frames}
x = \sum_{j=1}^m  \ip{x}{S^{-1} \phi_i}\phi_i,\text{ for all $x\in\mathcal{H}_n$,}
\end{equation}
hence, $\widetilde{\Psi}:=\seqsmall{S^{-1} \phi_i}_{i=1}^m$ is a dual
frame of $\Phi$ in the sense of Definition~\ref{def:dualframedef}, it
is called the \emph{canonical dual frame}. The canonical dual
$\widetilde{\Psi}$ has frame bounds $1/B$ and $1/A$, where $A$ and $B$
are frame bounds of $\Phi$.  In terms of matrices, $\widetilde\Psi$ is
the Moore-Penrose pseudo-inverse $\Phi^\dagger=\Phi^\ast (\Phi
\Phi^\ast)^{-1}$ of the analysis matrix $\Phi^\ast$.

\subsection{Sparsity and spectrum}
The measure we use for sparsity is the number of non-vanishing entries of the dual frame which we denote, with a slight abuse of notation, by $\|\cdot\|_0$. This convention as well as the associated name $\ell_0$-norm (although it is not a norm) is common practice for vectors and relates to the fact that it can be interpreted as the limit of $\ell_p$-quasinorms, $p\rightarrow 0^+$. In our paper, the term norm will also be used in a sense that includes the $\ell_0$-norm.

As for the spectral properties of a frame $\Phi$,  we will mostly work with singular values $\seqsmall{\sigma_i}_{i=1}^n$ of the synthesis matrix $\Phi$. While one
often studies the eigenvalues $\seqsmall{\lambda_i}_{i=1}^n$ of the frame operator $S_\Phi = \Phi \Phi^\ast$ in related cases, the
advantages of the singular value approach is that one then has the frame directly
accessible, whereby the construction of dual frames becomes, if not straightforward, then at least tractable. Obviously, the values are related by  $\lambda_i = (\sigma_i)^2$ for all $i=1,\dots,n$.

\subsection{The set of all duals}
\label{sec:dual-frames}

By the classical result in \cite{MR1374971}, all duals to a frame
$\Phi$ can be expressed as
\begin{equation}
 \set{S_\Phi^{-1}\phi_i + \eta_i - \sum_{k=1}^m \innerprod{S_\Phi^{-1}\phi_i }{\phi_k} \eta_k}_{i=1}^m,\label{eq:shidong-all-duals}
\end{equation}
where $\eta_i \in\mathcal{H}_n$ for $i=1,\dots,m$ is arbitrary; in matrix form it reads:
\[ 
S_{\Phi}^{-1}\Phi + E(I_m - \Phi^\ast S_{\Phi}^{-1}\Phi),
\]
where $E$ is a matrix representation of an arbitrary linear mapping from $\K^m$ to $\mathcal{H}_n$.  The canonical dual
appearing in these formulas can be replaced by any other dual, that is, if $\Psi$ is any dual, then all duals are given
by
\begin{equation}
 \Psi + E(I_m - \Phi^\ast \Psi). \label{eq:shidong-all-duals-op}
\end{equation}
The matrix $\Phi^\ast \Psi$ is the so-called \emph{cross Gramian} of the dual frame pair. Here $E:\K^m \to
\mathcal{H}_n$ can be an arbitrary linear mapping, but different choices of $E$ can yield the same dual, so there are less degrees of freedoms than entries in $E$. In fact, if one considers $\Phi$ as given and $\Psi$ as our unknown, equation~(\ref{eq:dual-eq-2}) corresponds to solving $n$ independent
linear systems, each of which consists of $n$ equations in $m$ variables. Hence, the set of all duals $\Psi$ to a frame $\Phi$ is an $n(m-n)$-dimensional affine subspace of $\mathrm{Mat}(\K,n \times m)$. 

A natural parametrization of this space is obtained using the singular value decomposition. Let $\Phi=U\Sigma_\Phi V^\ast$ be a full
SVD of $\Phi$, i.e.,  $U\in \K^{n \times n}$ and $V \in \K^{m \times m}$ are unitary and $\Sigma_\Phi \in \R^{n \times
  m}$ is a diagonal matrix whose entries are the singular values $\sqrt{B}= \sigma_1 \ge \sigma_2 \ge \dots \ge \sigma_n = \sqrt{A} >0$ of $\Phi$ as its diagonal entries $\sigma_j \ge 0$ in a non-increasing order. Let $\Psi$ be a frame and define $M_\Psi:=U^\ast \Psi V
\in \K^{n \times m}$, where $U$ and $V$ are the right and left singular vectors of $\Phi$. Then $\Psi$ factors as
$\Psi=U M_\Psi V^\ast$. By $\Phi
\Psi^\ast = \id[n]$, we then see that
\[ \id[n] = U^\ast \id[n] U = U^\ast \Phi \Psi^\ast U = \Sigma_\Phi M_\Psi^{\,\ast}. \]
Therefore, $\Psi$   is a dual frame of $\Phi$ precisely when
\begin{equation}
\Sigma_\Phi M_\Psi^{\,\ast} = \id[n], \label{eq:svd-dual-cond}
\end{equation}
where $\Psi=U M_\Psi V^\ast$. The solutions to \eqref{eq:svd-dual-cond} are given by
\begin{equation}
M_\Psi =
\begin{bmatrix}
  1/\sigma_1 & 0 & \cdots & 0 & s_{1,1} & s_{1,2} & \cdots & s_{1,r} \\
0& 1/\sigma_2 &  & 0 & s_{2,1} & s_{2,2} & \cdots & s_{2,r} \\
&  & \ddots  &  & \vdots & \vdots &  & \vdots \\
0& 0 & \dots  & 1/\sigma_n & s_{n,1} & s_{n,2} & \cdots & s_{n,r} \\
\end{bmatrix}\label{eq:svd-dual-parametric}
\end{equation}
where $s_{i,k}\in \K$ for $i=1,\dots,n$ and $k=1,\dots,r=m-n$. Note that the canonical dual frame is obtained by taking
$s_{i,k}=0$ for all $i=1,\dots,n$ and $k=1,\dots,m-n$.

\section{Sparsity of duals}
\label{sec:sparsity-duals}

The goal of this section is to find the sparsest dual frames of a given frame $\Phi\in\K^{n\times m}$. This question can be formulated as the minimization problem
\begin{equation} \label{eq:0-norm-min-problem}
  \min \norm[0]{\Psi} \quad \text{s.t.} \quad \Phi \Psi^\ast = I_n.
\end{equation}
\begin{definition}
We call the solutions of (\ref{eq:0-norm-min-problem}) \emph{sparsest duals} or \emph{optimal sparse duals}.
\end{definition}
 Let us
consider a small example illustrating that solutions to this minimization problem are not unique.
\begin{example}\label{example:sparsest-not-unique}
  We consider duals of the frame $\Phi=\seq{\phi_1,\phi_2,\phi_3}$ in $\R^2$ defined by
  \begin{equation}
\Phi =
\begin{bmatrix}[r]
  1 & -1 & 0 \\
  1 & 2  &-1
\end{bmatrix}.\label{eq:main-2x3-example}
\end{equation}
All duals $\Psi$ are easily found by Gauss-Jordan elimination to be given by
\begin{equation} 
\label{eq:elim-para-2-3-example}
 \Psi =
\begin{bmatrix}[r]
  \frac{2}{3} & -\frac13 & 0 \\[2pt]
  \frac13 & \frac13  & 0
\end{bmatrix}
+ t_1 \begin{bmatrix}[r]
  \frac{1}{3} & \frac13 & 1 \\[1pt]
  0 & 0  & 0
\end{bmatrix}
+ t_2 \begin{bmatrix}[r]
  0 & 0  & 0  \\
\frac{1}{3} & \frac13 & 1
\end{bmatrix}
, \quad t_1,t_2 \in \R ,
\end{equation}
hence the sparsest duals are seen to be 
\begin{equation}
 \Psi^1 =
\begin{bmatrix}[r]
  0 & -1 & -2 \\
  0 & 0  & -1
\end{bmatrix} , \quad
 \Psi^2 =
\begin{bmatrix}[r]
  \frac23 & -\frac13  & 0 \\[1pt]
  0 & 0  & -1
\end{bmatrix},
\quad \text{and} \quad
\Psi^3 =
\begin{bmatrix}[r]
  \frac23 & 0 & 1 \\
  0 & 0  & -1
\end{bmatrix}, \label{eq:2-by-3-example-sparsest-duals}
\end{equation}
which correspond to the parameter choices $t_2=-1$ and $t_1=-2,0,1$, respectively. These are the only three duals with
$\norm[0]{\Psi}=3$.  According to the measures derived in \cite{krahmer_Sparse_duals}, the non-degenerate dual $\Psi^2$ without zero vectors is preferred. This reflects the fact that it is the only one of the three duals for which one
is not discarding frame coefficients in the reconstruction procedure.
\end{example}

Solving (\ref{eq:0-norm-min-problem}) of sizes much larger than the
$2\times 3$ example in Example~\ref{example:sparsest-not-unique} is
not feasible since the problem is NP-complete. Hence to find sparse duals, one has to settle for approximate solutions of (\ref{eq:0-norm-min-problem}), \eg by convex relaxation or greedy strategies. For such approaches we refer to \cite{krahmer_Sparse_duals,shidong_l1-min-duals}. In the following subsections we will focus on analyzing the possible values of $\norm[0]{\Psi}$, rather than on  how one finds  minimizers. 
An exception is Section~\ref{sec:duals-spectr-tetr}, where finding the sparsest duals is tractable due to the specific structure of the frame.

\subsection{Upper and lower bounds for the minimal sparsity}
\label{sec:existence-results}
In this section we investigate the possible sparsity levels in the set of all dual frames. In other words, we consider
the possible objective function values of the minimization problem~(\ref{eq:0-norm-min-problem}). We start with a
trivial upper bound for the sparsity level of sparse duals.

\begin{lemma}\label{lem:sparse-dual-upper-bound}
  Suppose $\Phi$ is a frame for $\K^n$. Then there exists a dual frame $\Psi$ of $\Phi$ with
  \begin{equation}
 \norm[0]{\Psi} \le n^2.\label{eq:sparse-dual-upper-bound}
\end{equation}
\end{lemma}
\begin{proof}
  Choose $J \subset \{1,2,\dots, m\}$ such that $\card{J} =n$ and the vectors $\{\phi_i\}_{i\in J}$ are linearly
  independent. This is always possible, as otherwise, the $\phi_i$ do not span the space and hence do not form a frame. Let $\{\psi_i\}_{i \in J}$ be the unique (bi-orthogonal) dual of $\{\phi_i\}_{i\in J}$, and set $\psi_i=0$
  for $i \notin J$. Then obviously $\norm[0]{\Psi} \le n^2$ holds. 
\end{proof}

Similarly to Lemma~\ref{lem:sparse-dual-upper-bound} we see that the corresponding trivial lower bound on the dual frame
sparsity is $n$, that is, any dual frame satisfies $\norm[0]{\Psi}\ge n$. However, we can give a much better lower bound
in terms of the spark of the $\Phi^{(j)}$'s. Here the spark of a matrix $\Phi\in \mathrm{Mat}(\K,n \times m) $ is
defined as the smallest number of linear dependent columns of $\Phi$ and denoted by $\spark{\Phi}$. In case $\Phi$ is an
invertible $n \times n$ matrix, one sets $\spark{\Phi}=n+1$. 

\begin{theorem}
\label{thm:lower-bound-sparsity}
Suppose $\Phi$ is a frame for $\K^n$. Then any dual frame $\Psi$ of $\Phi$ satisfies
  \begin{equation}
 \norm[0]{\Psi} \ge \sum_{j=1}^n \spark{\Phi^{(j)}}.\label{eq:sparse-dual-lower-bound}
\end{equation}
\end{theorem}
\begin{proof}
  Fix $j =1,\dots,n$. By equation~(\ref{eq:dual-eq-2}), the $j$th row of $\Psi$, that is, $\psi^j$, must be orthogonal to $\phi^k$ for
  $k=1,\dots,j-1,j+1,\dots,n$. Therefore, we have that 
\begin{equation}  \label{eq:ortho-req}
0_{n-1}=\Phi^{(j)}(\psi^j)^\ast=[\phi^{(j)}_i]_{i\in\supp \psi^j}\, \psi^j \vert_{\supp
  \psi^j},
\end{equation}
where $0_{n-1}$ is the $(n-1)\times 1$ zero column vector.  Hence, the $\normsmall[0]{\psi^{j}}$ columns
$\setpropsmall{\phi^{(j)}_i}{j\in \supp{\psi^j}}$ of $\Phi^{(j)}$ are linearly dependent, which implies that
$\normsmall[0]{\psi^j} \ge \spark{\Phi^{(j)}}$ and thus the result.
\end{proof}

It is easy to see that the lower bound on the support given in Theorem~\ref{thm:lower-bound-sparsity} is not always
sharp. Indeed, if $\Phi$ has two linearly dependent columns, one has $\spark{\Phi^{(j)}}=2$ for all $j$, but there is
not necessarily a dual with only two entries in each row. In order to predict the maximally possible sparsity, one needs
a more refined notion that excludes this type of effects arising from linearly dependent columns. More precisely, denote
by $\spark[j]{\Phi}$ the smallest number of linearly dependent columns in $\Phi^{(j)}$ such that the corresponding
columns in $\Phi$ are linear independent. Obviously, we have that $\spark[j]{\Phi} \ge \spark{\Phi^{(j)}}$ for every
$j=1,\dots,n$.  Using this refined notation of spark we can improve the bound obtained in
Theorem~\ref{thm:lower-bound-sparsity} for optimal sparse duals. Even more, the following result exactly describes the
sparsity level of the sparsest dual.
\begin{proposition}
\label{thm:sparsity-of-sparsest-dual}
Suppose $\Phi$ is a frame for $\K^n$. Then the sparsest dual frame $\Psi$ of $\Phi$ satisfies
  \begin{equation}
 \norm[0]{\Psi} = \sum_{j=1}^n\spark[j]{\Phi}.
\end{equation}
\end{proposition}

\begin{proof}
  Let $\Psi$ be a sparsest dual of $\Phi$. Then $\bigl(\phi_k^{(j)}\bigr)_{k\in \supp \psi^j}$ must be linearly dependent in
  order to allow for $\Phi^{(j)} \left(\psi^j\right)^*=0_{n-1}$. However, the frame vectors $\seq{\phi_k}_{k\in \supp
    \psi^j}$ must be linearly independent, as otherwise one of these columns can be expressed through the others, which
  would allow for the construction of $\tilde \psi^j$ with $\supp\tilde\psi^j\subsetneq\supp \psi^j$ such that
  $\Phi\bigl(\tilde\psi^j \bigr)^\ast=e_j$. This in turn would imply that the frame, whose synthesis matrix has rows $\psi^1, \dots, \psi^{j-1}, \tilde\psi^j, \psi^{j+1}, \dots, \psi^n$, is also a dual frame of $\Phi$, so $\Psi$ is not the sparsest dual.
 Thus $|\supp \psi^j| \geq\spark[j]{\Phi}$,
  which implies that no dual frame can have less than $\sum_{j=1}^n\spark[j]{\Phi}$ non-zero entries.

  To show the existence of dual frames with this sparsity, let $S$ be a set of size $\spark[j]{\Phi}$ such that
  $\seq{\phi_k}_{k\in S}$ is a set of independent columns of $\Phi$ such that the corresponding columns of $\Phi^{(j)}$
  are linearly dependent. That is, there exist $\left(\lambda_k\right)_{k\in S}$ such that $\sum_{k\in S} \lambda_k
  (\phi^{(j)})_k =0$, but $(\sum_{k\in S} \lambda_k \phi_k)_j =a \neq 0$. Hence the vector $\psi^j$ given by
  \begin{equation}
    \psi^j_k=\begin{cases}
      \frac{\lambda_k}{a} &\text{ if } k\in S,\\
      0 & \text{ else},
       \end{cases}
   \end{equation}
   satisfies $\Phi (\psi^j)^\ast=e_j$, and the matrix $\Psi$ with rows $\psi^j$ yields a dual with a frame
   matrix support size of $\sum_{j=1}^n\spark[j]{\Phi}$, as desired.
\end{proof}

\subsection{Frames with the lower sparsity bound $n^2$}
\label{sec:ex-lower-bound-n^2}

In Lemma~\ref{lem:sparse-dual-upper-bound} we saw that it is always possible to find a dual frames with sparsity level
$n^2$. In this section we will show that for a large class of frames this is actually the best, or rather the sparsest,
one can achieve. 

Recall that a $n \times m$ matrix is said to be in general position if any sub-collection of $n$ (column) vectors is
linear independent, that is, if any $n \times n$ submatrix is invertible.  The following result states that if the
projection of $\Phi$ onto any of the $n$ coordinates hyperplanes of dimension $n-1$ is in general position, then the
frame has maximal dual sparsity level. This observation follows from
Theorem~\ref{thm:lower-bound-sparsity}. 
\begin{theorem}
\label{thm:general-position-has-all-duals-n-squared}
Suppose $\Phi$ is a frame for $\K^n$ such that the submatrix $\Phi^{(j)}$ is in general position for every
$j=1,\dots,n$. Then any dual frame $\Psi$ of $\Phi$ satisfies
  \begin{equation}
 \norm[0]{\Psi} \ge n^2.\label{eq:sparse-dual-lower-bound-for-GP}
\end{equation}
In particular, the sparsest dual satisfies $\norm[0]{\Psi} = n^2$
 \end{theorem}

\begin{proof}
  Fix $j =1,\dots,n$, and let us consider $c=(c_i)_{i=1}^m=(\psi^j)^\ast$ as a $m \times 1$ column vector to be chosen.
Then
   \begin{equation}
   \sum_{i=1}^m c_i \phi^{(j)}_i = 0_{n-1}. \label{eq:lin-dependency}
   \end{equation}
   Since $\Phi^{(k)}$ is in general position, any collection of $n-1$ vectors from $(\phi^{(k)}_i)_{i=1}^m$ is linearly
   independent. Hence, equation~(\ref{eq:lin-dependency}) can only be satisfied if $c$ has more than $n-1$ nonzero
   entries, that is, if $\card{\setprop{i}{c_i\neq 0}} \ge n$. Hence, each row of $\Psi$ must have at least $n$ nonzero
   entries, which yields a total of $n^2$ nonzero entries.
\end{proof}
We illustrate this result with a number of examples of frames which are well-known to be in general position, and which thus do not allow for dual frames with less than $n^2$ non-vanishing entries.

\begin{example}
\label{cor:n-and-m-with-all-duals-n-squared}
For any $n,m \in \N$ with $m\ge n$, let $a_i > 0$, $i =1, \dots, m$ with $a_i \neq a_j$ for all $i \neq j$, and let $b_j>0$, $j =1,
  \dots, n$ with $b_j \neq b_i$, $j \neq i$. We then define the generalized Vandermonde frames as
  \begin{equation}
   \Phi=
    \begin{bmatrix}
      a_1^{b_1} & a_1^{b_2} & \cdots & a_1^{b_m} \\
        \vdots & \vdots & & \vdots \\
      a_n^{b_1} & a_n^{b_2} & \cdots & a_n^{b_m} \\
    \end{bmatrix}.\label{eq:gen-vandermonde-frame}
  \end{equation}
It is not difficult to see that the submatrix $\Phi^{(j)}$ is in general position for every
   $j=1,\dots,m$.
This yields a deterministic construction for a frame $\Phi \in \K^{n \times m}$ such that any dual frame $\Psi$ of
$\Phi$ satisfies
\[ \norm[0]{\Psi} \ge n^2.\]
\end{example}

\begin{remark}
  We remark that  any minor of a generalized Vandermonde frames $\Phi$ is non-zero, which, in particular, implies that
  the submatrices $\Phi^{(j)}$ are in general position for every $j=1,\dots,m$. Neither of these properties hold, in
  general, for standard Vandermonde matrices. A standard Vandermonde frame $\Phi$ is obtained as the transpose of the
  matrix in \eqref{eq:gen-vandermonde-frame} with $b_j=j-1$ for $j=1,\dots,n$ and $a_i \in \K$ for
  $i=1,\dots,m$. Such a standard Vandermonde frame is again in general position, see also \cite{fullspark}. To see
  that this is not necessarily the case for all $\Phi^{(j)}$, take $n=3$ and let $a_{i_0}=-1$ and $a_{i_1}=1$ for some
  $i_0$ and $i_1 \in \set{1,\dots,m}$. Then the submatrix $\Phi^{(2)}$ restricted to the indices $\set{i_0,i_1}$ is a
  2-by-2 matrix of all ones, hence $\Phi^{(2)}$ cannot be not in general position.
\end{remark}

The generalized Vandemonde frames introduced above are almost never used in applications due to the fact that for $m\gg n$ they are necessarily
poorly conditioned, both in the traditional sense and in the frame theory sense as introduced in \cite{krahmer_Sparse_duals}. The following example yields well-conditioned, deterministic frames for which no dual frame with less than $n^2$ entries can exist either. 

\begin{example}
\label{thm:n-2-sparsity-HTF-m-prime}
Let $n \in \N$, and let $m$ be prime. Let $\Phi$ be a partial FFT matrix or a harmonic tight frame matrix of size $n
\times m$. Then any $\Phi^{(j)}$ is in general position, as the determinant of any
  $(n-1)\times (n-1)$ submatrix of $\Phi$ is non-zero. This is a consequence of Chebotarev theorem about roots of unity stating that any minor of an $m \times m$ DFT matrix is non-zero whenever $m$ is prime \cite{stevenhagen:1996, tao:2005, MR2299819}. 
Thus again, the sparsest dual frame $\Psi$ of $\Phi$ satisfies
\[ \norm[0]{\Psi} = n^2.\]
\end{example}


\begin{example}
 Let $n \in \N$ be prime and let $m=n^2$. It was shown in \cite{KraPfaRa08} that for almost every $a\in\C^n$, the Gabor frame generated by $a$ has the property that any minor of its frame matrix is non-zero. We conclude as before that the sparsest dual frame satisfies
\[ \norm[0]{\Psi} = n^2.\]
\end{example}

In fact, the property of having a sparsest dual with sparsity $n^2$ is a \emph{generic}
property. To precisely formulate this observation, let $\mathcal{F}(n,m)$ be the set of all
frames of $m$ vectors in $\mathcal{H}_n$, let $\mathcal{N}(n,m)$ be the set of all frames of $m$ vectors in $\mathcal{H}_n$ whose
sparsest dual has $n^2$ non-zero entries, and let $\mathcal{P}(n,m)$ be the the set of all frames $\Phi$ which
satisfy $\spark{\Phi^{(j)}}=n$ for all $j=1,\dots, n$. We remark that the following result can be obtained using techniques from algebraic geometry \cite{fullspark}; we include an elementary proof for illustrative purposes.

\begin{lemma}
  \label{thm:P-dense-and-large}
Suppose $m \ge n$. Then the following assertions hold:
\begin{compactenum}[(i)]
\item The set $\mathcal{P}(n,m)$ is open and dense in $\K^{n\times m}$, that is, $\mathcal{P}(n,m)^c$ is closed and
  nowhere dense.
\item The set $\mathcal{P}(n,m)^c$ is of (induced Euclidean) measure zero.
\end{compactenum}
\end{lemma}
\begin{proof}
  (i) Let $\Phi_0 \in \K^{n \times m}$ be a given. Pick $\Phi_1 \in \mathcal{P}(n,m)$, \eg a generalized Vandermonde
  frame. Define
  \[ 
  \Phi_{t} = t \Phi_1 + (1-t) \Phi_0 \qquad t \in \K.
  \] 
  We claim that only finitely many of $\setprop{\Phi_t}{t\in\K}$ are not in $\mathcal{P}(n,m)$. To see this we introduce
  the following polynomial in t:
\[
P(t) = \prod_{j=1}^n \prod_{I \in S} \det([\Phi^{(j)}_t]_{I}),
\]
where $S$ is the collection of all subsets of $\set{1,\dots,m}$ of size $n-1$ and $[\Phi^{(j)}_t]_{I}\in \K^{(n-1)\times
  (n-1)}$ is the matrix $\Phi_t$ without row $j$ and restricted to the columns in the index set $I$. Clearly,
$\card{S}=\binom{m}{n-1}$, and $P$ is therefore of degree $n^2\binom{m}{n-1}$. By definition, we see that
\begin{align*}
  & P(t) = 0 \\
\Leftrightarrow \;& \det([\Phi^{(j)}_t]_{I})=0 \quad \text{for some $I \in S$ and $j=1,\dots,n$} \\
\Leftrightarrow \;& \spark{\Phi^{(j)}_t} < n \quad \text{for some $j=1,\dots,n$} 
\end{align*}
Therefore, $P(t) \neq 0$ if and only if $\Phi_t \in \mathcal{P}(n,m)$. Since $P(1)\neq 0$, the polynomial $P(t)$ is not
the zero polynomial hence it has finitely many zeros, to be precise, it has at most $n^2\binom{m}{n-1}$
zeros. Therefore, we can choose $\abs{t}$ arbitrarily small such that $\Phi_t \in \mathcal{P}(n,m)$. Consequently,
$\Phi$ is arbitrarily close to a frame in $\mathcal{P}(n,m)$.

Part (ii) follows in a similar setup as in the prove of (i) using an integration argument.
\end{proof}

By Theorem~\ref{thm:lower-bound-sparsity} we see that $\mathcal{P}(n,m) \subset \mathcal{N}(n,m)$ and thus
$\mathcal{F}(n,m)\setminus \mathcal{N}(n,m) \subset \mathcal{F}(n,m)\setminus \mathcal{P}(n,m)$. Hence, the following
result immediately follows from Lemma~\ref{thm:P-dense-and-large}.
\begin{theorem}
  \label{thm:close-to-max-sparsity}
  Any frame in $\mathcal{F}(n,m)$ is arbitrarily close to a frame in $\mathcal{N}(n,m)$. Moreover, the set
  $\mathcal{F}(n,m)\setminus \mathcal{N}(n,m)$ is of measure zero.
\end{theorem}
Another consequence of Lemma~\ref{thm:P-dense-and-large} is that for many randomly generated frames, the sparsest dual has sparsity level $\norm[0]{\Psi} = n^2$. 
As an example, this holds when the entries of $\Phi$ are drawn independently at random from a standard normal distribution. Also frames obtained by a small Gaussian random perturbations of a given frame will have this property.

\subsection{Sparse duals of sparse frames}
\label{sec:duals-spectr-tetr}
We have seen in the previous sections that, while the sparsity of the sparsest dual of a frame can lie between $n$ and $n^2$, generically, its value is $n^2$. One can, however, by imposing structure, obtain classes of frames that allow for considerably sparser duals. One example is a class of very sparse frames, the so-called \emph{spectral tetris frames}.

The spectral tetris algorithm is constructs unit norm frames of $m$ vectors in $\R^m$ with prescribed eigenvalues of the frame
operator. It was developed in \cite{cfmwz11} and extended in \cite{cchkp}. The construction works for any sequence of
eigenvalues $\left(\lambda_j\right)_{j=1}^n$ that satisfy $\lambda_j \ge 2$ for all $j$ and for which $m:=\sum_{j=1}^{n}\lambda_j$ is an integer. 

We will need to introduce some notation that we will use throughout this section. For a sequence $\left(\lambda_j\right)_{j=1}^n$, we define $k_i$ and $\tilde m_i$ as follows. Set $k_0=0$ and $m_0=0$ and recursively choose $k_i=\inf\setpropsmall{\N\ni k > k_{i-1} }{ \tilde m_i:=\sum_{j=1}^{k}\lambda_j \in \N}$. Naturally, the largest value $\mu$ that can be chosen for $i$ is such that $\tilde m_\mu=m$ and $k_\mu=n$.  Furthermore, let $K:=\setprop{k_i}{i=0,1,\dots,\mu}$.
\begin{definition}
\label{dfn:STF}
For a sequence $\left(\lambda_j\right)_{j=1}^n$ satisfying the condition $m:=\sum_{j=1}^n \lambda_j \in \N$, the Spectral Tetris
frame $\STF(n,m,\set{\lambda_j}) \in \R^{n\times m}$ is given by
\[
\begin{bmatrix}
   1\;  \cdots\;  1 & a_1 & a_1 & & & & & & &      \\
   &   b_1 & -b_1 & 1 \;  \cdots\; 1  & a_2 & a_2 & & & & &    \\
   &  & & & b_2 & -b_2 &  1  \;   \cdots   & &   \\
    & & & &  &  &  &   &  & &     \\
   & & & &   &   & &   \cdots \; 1  & a_{n-1} & a_{n-1} &   \\
     \multicolumn{1}{c}{\myunderbrace{\phantom{1 \; \cdots \;
           1}}{m_1 \text{ times}}} &  & &
     \multicolumn{1}{c}{\myunderbrace{\phantom{1 \; \cdots \;
           1}}{m_2 \text{ times}}} & & & &    & b_{n-1} & -b_{n-1} & \multicolumn{1}{c}{\myunderbrace{1 \; \cdots \;
           1}{m_n \text{ times}}} 
\end{bmatrix},\medskip
\]
where, for $j\notin K$, $a_j:=\sqrt{\frac{r_j}{2}}$ and $b_j:=\sqrt{1-\frac{r_j}{2}}$, and $r_j \in \itvco{0}{1}$ and $m_j \in \N_0$ are defined by
\begin{align}
  \lambda_j &= m_j + r_j, && \text{when $j-1\in K$} ,
  \label{eq:m-and-r-in-K}  \\
  \lambda_j &= (2-r_{j-1}) + m_j + r_j, && \text{otherwise}. \label{eq:m-and-r--not-in-K}
\end{align}
If $j\in K$, the $2\times 2$-block matrix $B_j=\left[
  \begin{smallmatrix}
    a_j & a_j \\ b_j & -b_j
  \end{smallmatrix} \right]$ is omitted.
\end{definition}
%

\begin{proposition}
\label{thm:sparsity-STF}
  Suppose $\left(\lambda_j\right)_{j=1}^n$ satisfies $\lambda_j \ge 2$, $j=1,\dots,n$, and $m:=\sum_{j=1}^n\lambda_j
  \in \N$. Let $\Phi:=\STF(n,m,\left(\lambda_j\right)_{j=1}^n)$. Define the set $I,J \subset \set{1,\dots,n}$ by
\begin{equation*}
   \label{eq:def-of-I}
I = \{0\leq j\leq \mu-1: k_{j+1}=k_j+1 \} \qquad \text{and}
\end{equation*}
\begin{equation*}
  \label{eq:def-of-J}
  J = \setprop{j_0 \in \itvoo{k_i+1}{k_{i+1}}}{i = 0,\dots, \mu \text{ and } \sum_{j=k_i+1}^{j_0}\lambda_j -
    \left( \floor{\sum_{j=k_i+1}^{j_0-1}\lambda_j} +2 \right) \ge 1 },
\end{equation*}
and let $\hat k :=2\mu + \card{J}-\card{I}$, where $\mu$ is the maximal index $i$ as above. Then the sparsest dual $\Psi$ satisfies
\[
\norm[0]{\Psi}= \hat k+2(n-\hat k).
\]
\end{proposition}

\begin{proof}
  Let $\seqsmall{e_j}_{j=1}^n$ denote the standard orthonormal basis of $\R^n$. If $e_j \in \Phi$, then $\spark[j]{\Phi}=1$
  since $\Phi^{(j)}$ contains a zero column. On the other hand, if $e_j \not\in \Phi$, then $\spark[j]{\Phi}=2$. Let
  \begin{equation}
 \tilde k = \card{\setprop{j}{e_j \in \Phi}} .\label{eq:claim-on-k}
  \end{equation}
  By Proposition~\ref{thm:sparsity-of-sparsest-dual}, it then follows directly that the sparsest dual satisfies
  \[
  \norm[0]{\Psi} = \sum_{j=1}^n\spark[j]{\Phi}=\tilde k+2(n-\tilde k).
  \]
  To complete the proof, we need show that $\tilde k$ always equals $ \hat k=2\mu+\card{J}-\card{I}$, i.e.,  that
  there are exactly $2\mu+\card{J}-\card{I}$ rows of $\Phi$ containing a $1$.

  For that, recall that each $k_j$ corresponds to a $2\times 2$-block matrix $B_j=\left[
  \begin{smallmatrix}
    a_j & a_j \\ b_j & -b_j
  \end{smallmatrix} \right]$ being omitted at the transition from one row to the next. Hence the $k_j$th column is $e_ {m_j}$, the $(k_j+1)$st column is $e_{m_j+1}$; each of the $\mu$ values of $k_j$ yields two rows that contain an $e_j$.
if $j \in \bigl(K \cup (K+1) \bigr) \cap \set{1,\dots, n}$, then $e_j \in\Phi$. Correcting for the double counting resulting from the case $k_{j+1}=k_j+1$, there are $2\mu-\card{I}$ values of $\ell$ such that $e_\ell \in \Phi$ resulting from this phenomenon.

In addition, there can be $e_\ell\in \Phi$ that do not correspond to such a transition, i.e., that appear as the $j_0$th column of the frame matrix, where $j_0$ is such that $k_i+1 < j_0 < k_{i+1}$ for some $i$. For such a $j_0$ we calculate
  \begin{align*}
    m_{j_0} + r_{j_0} &= \lambda_{j_0} - 2 + r_{j_0-1} \\
    &= \lambda_{j_0} - 2 + \lambda_{j_0-1} - 2 + r_{j_0-2} - m_{j_0-1}
    \\
    &= (\lambda_{j_0} + \lambda_{j_0-1}) - (4 + m_{j_0-1}) - r_{j_0-2}
    \\
    &= \cdots = \sum_{j=k_i+1}^{j_0}\lambda_{j} - \Bigl(2(j_0-1) + \sum_{j=k_i+1}^{j_0-1}m_{j}\Bigr)
  \end{align*}
  By definition the $j_0$th row contains a $1$ if and only if $m_{j_0}\ge 1$. Furthermore, it can easily be shown, \eg
  by induction, that
  \[
  2(j_0-k_i-1) + \sum_{j=k_i+1}^{j_0-1}m_{j} = \floor{\sum_{j=k_i+1}^{j_0-1}\lambda_j} +2.
  \]
  Therefore, the $j_0$th row of $\Phi$ contains a $1$, if and only if
  \[
  \sum_{j=k_i+1}^{j_0}\lambda_{j} - \left(\floor{\sum_{j=k_i+1}^{j_0-1}\lambda_j} +2\right) \ge 1.
  \]
This is just the defining condition of $J$, so one has an additional number of $\card{J}$ vectors $e_\ell\in\Phi$, which proves the theorem.
\end{proof}

\begin{remark}
  Since $\mu \ge 1$, we see that $\hat k\ge 2$. Thus the sparsest dual of a Spectral Tetris frame satisfies $n \le \norm[0]{\Psi}=2n-\hat k\le
  2n-2$. Furthermore, if $\lambda_j \ge 3$ for all $j=1,\dots, n$, we see that $\norm[0]{\Psi}=n$ since $\hat k=n$.
\end{remark}

\begin{remark}
  Recall that the spectral tetris algorithm is only guaranteed to work
  for redundancies larger than two, see Definition~\ref{dfn:STF}. In
  \cite{lemvig_prime} the proof of Proposition~\ref{thm:sparsity-STF}
  is used to characterize when the spectral tetris construction works
  for tight frames with redundancies below two. In fact, a similar argument works for general frames without a tightness condition.
\end{remark}
\begin{remark}
 A construction scheme for optimal sparse duals of Spectral Tetris frames directly follows from the proof of Proposition~\ref{thm:sparsity-STF}. The resulting dual will consist only of $1$-sparse vectors and $2$-sparse vectors.
\end{remark}
%

\section{Spectral properties of duals}\label{sec:spectrum-dual}

In this section we answer the question of which possible spectra the set of all dual frames admits. A special case of this question has already been answered in \cite{massey_optimal-dual-frames}, namely the question, which frames admit tight duals. In linear algebra terms this boils down to the question whether a full-rank matrix $\Phi$ has a right inverse with condition number one. 
The first step towards a more general characterization will be a more quantitative version of their result. In particular,  we study which frame bounds tight duals can attain for a given frame. Furthermore, our proof provides an explicit construction procedure for these duals. The understanding developed in the proof will turn out useful in some more general cases.

It turns out that a frame always has a tight dual if the redundancy is two or larger. If the
redundancy is less than two, it will only be possible under certain
assumptions on the singular values of $\Phi$. 
Note that in
\cite{massey_optimal-dual-frames}, the result is formulated in terms of the eigenvalues of the frame operator, so even the existence part of the following result is formulated slightly different from the
result in \cite{massey_optimal-dual-frames}.
\begin{theorem}
\label{thm:tight-dual}
  Let $n,m \in \N$. Suppose $\Phi$ is a frame for $\K^n$ with $m$ frame vectors and lower frame bound $A$. Then the following assertions hold:
  \begin{compactenum}[(i)]
  \item If $m \geq 2n$, then for every $c\geq \frac{1}{\sqrt A}$, there exists a tight dual frame $\Psi$ with frame bound $c$.
  \item If $m = 2n-1$, then there exists a tight dual frame $\Psi$; the only possible frame bound is $\frac{1}{\sqrt A}$.
  \item Suppose $m < 2n-1$. Then there exists a tight dual frame $\Psi$ if and only if the smallest $2n-m \in \set{2,\dots,n}$ singular values of $\Phi$ are equal. 
  \end{compactenum}
\end{theorem}

\begin{proof}
Let $\Phi=U\Sigma_\Phi V^\ast$ be a full SVD of $\Phi$, and let $\Psi$ be an arbitrary
  dual frame. Following Section~\ref{sec:dual-frames}, we factor the dual frame as $\Psi=U M_\Psi V^\ast$, where
  $M_\Psi$ is given as in \eqref{eq:svd-dual-parametric} with $s_{i,k}\in \K$ for $i=1,\dots,n$ and
  $k=1,\dots,r=m-n$. For $\Psi$ to be tight, we need to choose $s_{i,k}$ such that the rows of $M_\Psi$ are
  orthogonal and have equal norm. This follows from the fact that $\Psi$ is row orthogonal if and only if $M_\Psi$
  is row orthogonal.

  As the diagonal block of $M_\Psi$ is well-understood, the duality and tightness constraints translate to conditions for the inner products of the $s_i=(s_{i,1},\dots,s_{i,r})\in \K^r$, $i=1,\dots,n$. 
 Indeed, $\Psi$ is a tight dual frame with frame bound $c$ if and only if, for all $1\leq i\leq n$, one has
 \begin{equation}
 c= \frac{1}{\sigma_i^2}+\norm[2]{s_i}^2, \label{eq:row-norm-cond}
\end{equation}
and, for all $i\neq j=1,\dots,n$, one has 
$\innerprods{s_i}{s_j}=0$.

Now assume that $\sigma_n=\sigma_{n-1}=\dots=\sigma_{p+1} < \sigma_p$ for some $p <n$. The case of distinct singular values corresponds to $p=n-1$.
As $\sigma_{p+1} <
  \sigma_{i}$ for all $1\leq i\leq p$, \eqref{eq:row-norm-cond} implies that all $s_i$ for $i=1,\dots,p$ must be nonzero
vectors even if $s_{p+1}, \dots,s_n$ are all zero. Furthermore, \eqref{eq:row-norm-cond} also determines the norms of $s_1,\dots,s_p$ as a function of $\normsmall{s_{p+1}}=\dots =\norm{s_n}$. The second condition implies that if $s_n\neq 0$, the sequence $\seqsmall{s_j}_{j=1}^n$ is orthogonal, else the sequence $\seqsmall{s_i}_{i=1}^p$.

If $r\geq n$, that is, if $m\geq 2n$, then any choice of $s_n$ allows for an orthogonal system with compatible norms, so tight dual frames with any frame bound above $\frac{1}{\sigma_n}$ exist and can be efficiently constructed. If $r<n$, then no $n$ vectors can form an orthogonal system, one needs to have $s_n=0$ and hence also $s_j=0$ for all $j>p$. So no frame bound other than $\frac{1}{\sigma_n}$ is possible. The remaining vectors $\setsmall{s_j}_{j=1}^p$ are all non-zero, so they must form an orthogonal system. For $r\ge n-p+1$, this is possible, and again a solution satisfying the norm constraints can be efficiently constructed. For $r\leq n-p$, no such system exists, hence there cannot be a tight dual.

 This completes the proof.
\end{proof}

 We remark that for tight frames $\Phi$, hence $p=0$, one has the following two cases: If $m< 2n$, then $\Phi$ has only $A^{-1}\Phi$ as tight dual frame. If $m\ge 2n$, then $\Phi$ has infinitely many tight dual frames, in particular, $\Phi$ has a $C$-tight dual frame for any $C \ge A^{-1}$.

We will now derive general conditions on which spectral patterns (now possibly consisting of more than one point) can be achieved by a dual frame of a given frame. The reason that, in the general framework, such an analysis is harder than in the context of tight duals is that in that case, the frame operator is a multiple of the identity, hence diagonal in any basis. This no longer holds true if we drop the tightness assumption, so when the orthogonality argument of Theorem~\ref{thm:tight-dual} fails, one cannot conclude that there is no dual with a given spectral pattern. However, the orthogonality approach allows to choose a subset of the singular values of the dual frame freely. In particular, 
if the redundancy of the frame $\Phi$ is larger than $2$, it follows that for all spectral patterns satisfying a set of lower bounds, which we will later show to be necessary (see Theorem~\ref{thm:interlacing-sing-values}), a dual with that spectrum can be found using a constructive procedure analogous to the proof of Theorem~\ref{thm:tight-dual}.
\begin{theorem}
\label{thm:possible-sing-values-of-dual}
Let $n,m \in \N$, and let $\Phi$ be a frame for $\K^n$ with $m$ frame vectors and singular values
$\seqsmall{\sigma_i}_{i=1}^n$.  Suppose that $r \le m-n$ and that $I_r \subset \{1,\dots,n\}$ with $\card{I_r}=r$. Then, for any sequence $\seqsmall{q_i}_{i\in I_r}$
satisfying $q_i\ge  1/\sigma_i$ for all $i \in I_r$, there exists a dual frame $\Psi$ of $\Phi$
such that  $\setsmall{q_i}_{i\in I_r}$ is contained in the spectrum of $\Psi$. Furthermore, it can be found constructively using a sequence of orthogonalization procedures.
\end{theorem}

\begin{proof}
The proof is just a slight modification of the proof of Theorem~\ref{thm:tight-dual}. Again, we choose $(s_i)_{i \in I_r}$ to be orthogonal and the remaining $s_i$'s to be the zero vector. The non-zero $s_i$ vectors are scaled to satisfy 
  \begin{equation*}
    q_i^2= \frac{1}{\sigma_i^2}+\norm[2]{s_i}^2,
  \end{equation*}
where $i \in I_r$. Hence, by this procedure we obtain a dual frame with spectrum $\setsmall{q_i}_{i \in I_r} \cup \setsmall{\sigma^{-1}_i}_{i \notin I_r}$.
\end{proof}

As a corollary we obtain that using the same simple constructive procedure, one can find dual frames with any frame bound that is possible.
\begin{corollary}
  Let $\Phi$ be a redundant frame for $\K^n$ with singular values $\seqsmall{\sigma_i}_{i=1}^n$. Fix an upper frame bound satisfying $B^\Psi\ge \frac{1}{\sigma_{n}^2}$ and a lower frame bound 
$\frac{1}{\sigma_{m-n+1}^2} \ge A^\Psi \ge  \frac{1}{\sigma_{1}^2}$,
where we use the convention $\frac{1}{\sigma_{m-n+1}}=\infty$ if $m \ge 2n$. Then a dual frame $\Psi$ of $\Phi$ with these frame bounds can be found constructively using a sequence of orthogonalization procedures.
\end{corollary}

To obtain a complete characterization on the possible spectra for dual frames, we will need the following two results by
Thompson \cite{Th72} on the interlacing properties of the spectrum of submatrices.  To simplify notation in the following proof, we use a simpler formulation than in \cite{Th72}, which is technically speaking a special case of the setup given in \cite{Th72}, but up to permutation captures the results in complete generality.

For our setup, let $A\in\K^{n\times m}$ with singular values $\alpha_1\geq \dots \alpha_{\min(n,m)}$,  where $m,n\in\N$.
Furthermore, for $q,\ell \in\N$, $\ell \ge q$, let $P_q:\K^\ell\rightarrow\K^q$ denote the restriction on the $q$ first entries. Of course, the $P_q$ depends not only on $q$, but also on the input dimension $\ell$, but as this can usually be inferred from the context, we suppress this dependence to simplify notation. Note that $P^*_q P_q$ is the projection onto $\Span\setpropsmall{e_i}{i=1,\dots,q}$, \ie
\[
(P^*_q P_q
x)_i=\begin{cases}
  x_i \quad\text{if }i\leq q, \\
  0\quad\text{ else,}  \end{cases}\]
and $P_q P_q^*$ is the identity on $\K^q$. Furthermore, we take $p,q \in \N$ such that $p \le n$ and $q \le m$.
\begin{theorem}[{\cite[Theorem 1]{Th72}}]
\label{thm:interlacing1}
 Suppose $ \beta_1\geq \dots \geq \beta_{\min(p,q)}$ are the singular values of $B \in \K^{p\times q}$ given by $B= P_p A P^*_q$. Then
 \begin{align}
  &\alpha_i \geq \beta_i &\text{for \quad}&1\leq i\leq \min\set{p,q}, \label{eq:interlacing1} \\
  &\beta_i \geq \alpha_{i+(m-p)+(n-q)} &\text{for \quad }&i\leq \min\set{p+q-m,p+q-n}. \label{eq:interlacing2}
 \end{align}
\end{theorem}

\begin{theorem}[{\cite[Theorem 2]{Th72}}]
\label{thm:interlacing2}
 Let $ \beta_1\geq \dots \geq \beta_{\min(p,q)}$ satisfy \eqref{eq:interlacing1} and \eqref{eq:interlacing2}. Then there exist unitary matrices $U\in\K^{n\times n}$ and $V\in\K^{m\times m}$ such that $P_p U A V^\ast P^*_q \in \K^{p\times q}$ has singular values $\seqsmall{\beta_i}_{i=1}^{\min\{p,q\}}$.
\end{theorem}

These two results now allow for a complete characterization of the possible spectra of dual frames.
\begin{theorem}
\label{thm:interlacing-sing-values}
Let $n,m \in \N$, and set $r=m-n$. Let $\Phi$ be a frame for $\K^n$ with singular values
$\seqsmall{\sigma_i}_{i=1}^n$. 
Suppose $\Psi$ is any dual frame with singular values $\seqsmall{\sigma_i^\Psi}_{i=1}^n$ (also arranged in a non-increasing order). Then the following inequalities hold:
\begin{align}
  \frac{1}{\sigma_{n-i+1}} &\le  \sigma^\Psi_i &&\text{for } i=1,\dots,r,    \label{eq:interlacing-ineq-sing1}
\\ 
  \label{eq:interlacing-ineq-sing2}
  \frac{1}{\sigma_{n-i+1}} &\le  \sigma^\Psi_i \le \frac{1}{\sigma_{n-i+r+1}} &&\text{for } i=r+1,\dots,n.
\end{align}
Furthermore, for every sequence $\seqsmall{\sigma_i^\Psi}_{i=1}^n$  which satisfies \eqref{eq:interlacing-ineq-sing1} and \eqref{eq:interlacing-ineq-sing2}, there is a dual $\Psi$ of $\Phi$ with singular values $\seqsmall{\sigma_i^\Psi}_{i=1}^n$.
\end{theorem}

\begin{proof}
To show the necessity of the conditions, we will apply Theorem~\ref{thm:interlacing1} for $p=q=n$. Note that this entails that $P_p=I_n$.
Let $\Phi=U\Sigma_\Phi V^\ast$ be a full SVD of $\Phi$. We factor dual frames of $\Phi$ as $\Psi=U M_\Psi V^\ast$, where
  $M_\Psi$ is given as in \eqref{eq:svd-dual-parametric} with $s_{i,k}\in \K$ for $i=1,\dots,n$ and
  $k=1,\dots,r=m-n$. We let $M_{\widetilde{\Psi}}$ denote the case where we choose $s_{i,k}=0$ for all $i,k$; in this case $\widetilde{\Psi}=U M_{\widetilde{\Psi}} V^\ast$ is the canonical dual. Note that $M_{\Psi} P^*_n \in \K^{n\times n}$ has the same spectrum as $M_{\Psi} P_n^*P_n= M_{\widetilde{\Psi}}$, namely singular values $\seqsmall{1/\sigma_{n-i+1}}_{i=1}^n$. Furthermore, $\Psi$ and $M_\Psi$ have the same singular values  $\seqsmall{\sigma_i^\Psi}_{i=1}^n$. A direct application of Theorem~\ref{thm:interlacing1} for $A=M_\Psi$ and $B=M_{\Psi}P^\ast_n\in \K^{n\times n}$ now shows \eqref{eq:interlacing-ineq-sing1} and \eqref{eq:interlacing-ineq-sing2}.

For the existence, we need to show that there are unitary matrices $U\in \K^{n\times n}$, $V\in K^{m\times m}$ such that for the diagonal matrix $\Sigma_\Psi\in \R^{n\times m}$ with diagonal entries $\seqsmall{\sigma_i^\Psi}_{i=1}^n$ one has 
\begin{equation}
 \Phi V \Sigma_\Psi^* U^* =I_n. \label{eq:dualsvd}
\end{equation}
For this we will apply Theorem~\ref{thm:interlacing2} for $A= \Phi$, $\seqsmall{\alpha}_{i=1}^n = \seqsmall{\sigma_i}_{i=1}^n$, and again $p=q=n$. 
Let a sequence $\seqsmall{\sigma_i^\Psi}_{i=1}^n$ which satisfies \eqref{eq:interlacing-ineq-sing1} and \eqref{eq:interlacing-ineq-sing2} be given. Note that for our choice of parameters, this is equivalent to having $\seqsmall{\beta_i}_{i=1}^n = \seqsmall{1/\sigma_n^\Psi,\dots,1/\sigma_1^\Psi}$ satisfy \eqref{eq:interlacing1} and \eqref{eq:interlacing2}. Hence, by Theorem~\ref{thm:interlacing2}, there exist unitary matrices $U_1 \in \K^{n\times n}$ and $V_1\in \K^{m\times m}$ such that $U_1 \Phi V_1^* P_n^*$ and hence $ \Phi V_1^* P_n^*$ have singular values $\seq{(\sigma_i^\Psi)^{-1}}_{i=1}^n$ (in a non-decreasing order). By yet another SVD, there exist unitary matrices $U,V_2\in \K^{n\times n}$  such that $U^*\Phi V_1^* P^*_n  V_2^*$ is an $n\times n$ diagonal matrix with entries $\seq{(\sigma_i^\Psi)^{-1}}_{i=1}^n$ (note that this time, the variant of the SVD with singular values in a non-decreasing order is used). Thus, we have
\begin{equation}
 U^* \Phi V_1^* P^*_n  V_2^* P_n \Sigma_\Psi^*  = I_n
\end{equation}
and, noting that $(I_m-P_n^*P_n) \Sigma^*_\Psi =0$, this leads to
\begin{equation}
 \Phi V_1^*(I_m-P_n^*P_n+  P^*_n  V_2^* P_n) \Sigma_\Psi^* U^* = I_n.
\end{equation}
A direct calculation using the fact that $P_n P_n^*=I_n$ shows that 
\begin{equation}
V=V_1^*(I_m-P_n^*P_n+  P^*_n  V_2^* P_n) \in \K^{m\times m}
\end{equation}
is unitary, so we obtain \eqref{eq:dualsvd} as desired.
\end{proof}

The inequalities \eqref{eq:interlacing-ineq-sing1} and \eqref{eq:interlacing-ineq-sing2}, written in terms of the singular values $(\sigma^{\widetilde{\Psi}}_i)_{i=1}^n$ of the canonical dual frame $\widetilde{\Psi}:=S^{-1}\Phi$, have the following simple form:  
 \begin{align*}
  \sigma^{\widetilde{\Psi}}_i &\le  \sigma^\Psi_i &&\text{for } i=1,\dots,r,  \\ 
  \sigma^{\widetilde{\Psi}}_i &\le  \sigma^\Psi_i \le \sigma^{\widetilde{\Psi}}_{i-r} &&\text{for } i=r+1,\dots,n.
\end{align*}
We remark that the first part of Theorem~\ref{thm:interlacing-sing-values} also follows from $r$ applications of 
\cite[Theorem 7.3.9]{MR0832183} on the matrix $M_\Psi$ defined in
  (\ref{eq:svd-dual-parametric}) or from the well-known interlacing inequalities for Hermitian matrices by Weyl. 

Also note that the proof of Theorem~\ref{thm:interlacing2} and hence the proof of Theorem~\ref{thm:interlacing-sing-values} involves existence results for unitary matrices from \cite{Th66} and is, consequently, less intuitive and constructive than our above proof for the special case in Theorem~\ref{thm:possible-sing-values-of-dual}.  For this reason, we decided to include Theorem~\ref{thm:possible-sing-values-of-dual}, even though it directly follows from the general existence result of Theorem~\ref{thm:interlacing-sing-values}.

In terms of eigenvalues of frame operators, Theorem~\ref{thm:interlacing-sing-values} states that the spectra in the set of all duals exhaust the set $\Lambda \subset \R^n$ defined by 
\[ \Lambda=\setprop{(\lambda_1,\dots,\lambda_n)\in \R^n}{\lambda^{\widetilde{\Psi}}_i \le \lambda_i \le \lambda^{\widetilde{\Psi}}_{i-r} \text{ for all
    $i=1,\dots,n$}},\] where $\lambda^{\widetilde{\Psi}}_i = 1/\lambda^\Phi_{n-i+1}$ is the $i$th eigenvalue of the canonical dual frame
operator; we again use the convention that $\lambda^{\widetilde{\Psi}}_i=\infty$ for $i\le 0$. 
By considering the trace of $M_\Psi M_\Psi^\ast$, 
we see that the canonical dual frame is the unique dual frame that minimizes the inequalities in $\Lambda$. Any other spectrum in $\Lambda$ will not be associated with a unique dual frame, in particular, if $s_i=(s_{i,1},\dots,s_{i,r}) \neq 0$  in $M_{\Psi}$ for some $i=1,\dots,n$, then replacing $s_i$ by $zs_i$ for any $\abs{z}=1$ will yield a dual frame with unchanged spectrum. 
The frame operator $S_{\Psi}$  and the canonical dual frame operator $S_{\widetilde{\Psi}}$ are connected by
$S_\Psi = S_{\widetilde{\Psi}} + C$,
where $C=W^\ast W$, and $W\in \K^{r \times n}$ is arbitrary. Hence, in linear algebra terms, Theorem~\ref{thm:interlacing-sing-values} tells us that 
for a given positive definite, Hermitian matrix $S_{\widetilde{\Psi}}$ with spectrum $\seqsmall{\lambda^{\widetilde{\Psi}}_i}_{i=1}^n$ and for a given sequence $\seqsmall{\lambda_i}_{i=1}^n \in \Lambda$, there exists a positive semi-definite, Hermitian matrix $C$ of rank at most $\min\set{m-n,n}$ such that the spectrum of $S_\Psi := S_{\widetilde{\Psi}} + C$ consists of $\seqsmall{\lambda_i}_{i=1}^n$.

For a better understanding of the more general framework where Theorem~\ref{thm:possible-sing-values-of-dual} does not yield a complete characterization of the possible spectral patterns, we will continue by an extensive discussion of the example of a frame of three vectors in $\R^2$.


\begin{example}
\label{example:2-by-3-spectral}
   Suppose $\Phi$ is a frame in $\R^2$ with $3$
  frame vectors and frame bounds $0 < A^\Phi \le B^\Phi$, and let
  $\Phi=U \Sigma_\Phi V^\ast$ be the SVD of $\Phi$.  Then all dual
  frames are given as $\Psi=U M_\Psi V^\ast$, where
  \[ M_\Psi =
  \begin{bmatrix}
    1/\sigma_1 & 0 & s_{1} \\
    0 & 1/\sigma_2 & s_{2}
  \end{bmatrix} \] for $s_1,s_2 \in \R$. Since the frame operator of
  the dual frame is given by $S_{\Psi}=\Psi \Psi^\ast = U M_\Psi
  M_{\Psi}^\ast U^\ast$, we can find the eigenvalues of $S_\Psi$ by
  considering eigenvalues of
  \[ S:=M_\Psi M_{\Psi}^\ast = \begin{bmatrix}
    1/\sigma_1^2+s_1^2 & s_{1}s_2 \\
    s_1s_2 & 1/\sigma^2_2+s_2^2
  \end{bmatrix}. \] These are given by
  \[ \lambda_{1,2} = \frac{1}{2} \tr{S} \pm \frac{1}{2}R,\quad
  \text{where} \quad R = \sqrt{(\tr{S})^2-4\det{S}}.\] One easily sees
  that $\tr{S}$ monotonically grows as a function of
  $s_1^2+s_2^2$, whereas for fixed $\tr{S}$, the term $R$ grows as a function
  of $s_1^2-s_2^2$.  This exactly yields the two degress of freedom
  predicted by the existence part of
  Theorem~\ref{thm:interlacing-sing-values}.  A straightforward
  calculation shows that $R+(s_1^2+s_2^2) \ge
  \tfrac{1}{\sigma_2^2}-\tfrac{1}{\sigma_1^2}\ge 0$, hence we see that
  \begin{equation}
    \lambda_1 \ge \tfrac{1}{\sigma_2^2} \quad \text{and} \quad \tfrac{1}{\sigma_2^2} \ge\lambda_2 \ge \tfrac{1}{\sigma_1^2},\label{eq:ineq-example}
  \end{equation}
  which is also the conclusion of the necessity part of
  Theorem~\ref{thm:interlacing-sing-values}. We remark that the two
  eigenvalues depend only on quadratic terms of the form $s_1^2$ and
  $s_2^2$. Therefore, if $s_1$ and $s_2$ are non-zero, then the
  choices $(\pm s_1,\pm s_2)$ yield four different dual frames having
  the same eigenvalues. In this case the level sets of $\lambda_1$ as
  a function of $(s_1,s_2)$ are origin-centered ellipses with major
  and minor axes in the $s_1$ and $s_2$ direction,
  respectively. Moreover, the semi-major axis is always greater than
  $s_0:=(\sigma_2^{-2}-\sigma_1^{-2})^{1/2}$. The level sets of
  $\lambda_2$ are origin-centered, East-West opening hyperbolas with
  semi-major axes greater than $s_0$. In
  Figure~\ref{fig:dual-frame-bounds} the possible eigenvalues of the
  dual frame operator of the frame $\Phi$ defined by
  \begin{equation} \label{eq:new-2x3-example} \Phi =
    \frac{1}{50}\begin{bmatrix}
      90 & -12 & -16 \\
      120 & 9 & 6
    \end{bmatrix}
  \end{equation}
  are shown as a function of the two
  parameters $s_1$ and $s_2$; Figure~\ref{fig:dual-frame-bounds-b}
  shows the level sets and the four intersection points $(\pm s_1,\pm
  s_2)$ for each allowed spectrum in the interior of $\Lambda$.
   Note that the singular values are $\sigma_1=3$ and $\sigma_2=1/2$, hence
  $B^\Phi=9$ and $A^\Phi=1/4$.

\begin{figure}[ht!]
  \centering \subfloat[Graphs of $\lambda_{1}$ and $\lambda_{2}$]{%
    \includegraphics[height=0.45\textwidth]{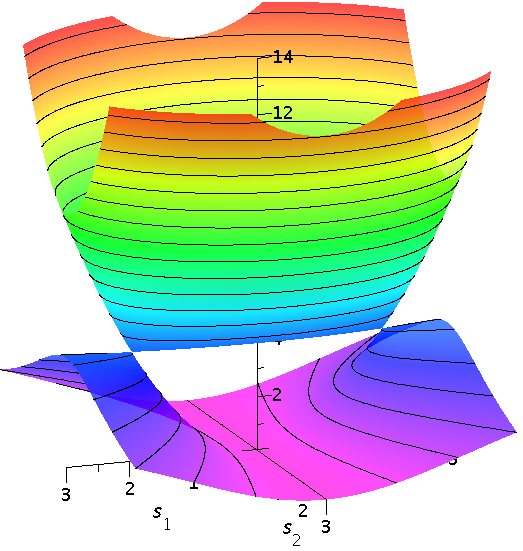}
    \label{fig:dual-frame-bounds-a}} \subfloat[Level curves of
  $\lambda_{1}$ and $\lambda_{2}$]{%
    \includegraphics[height=0.45\textwidth]{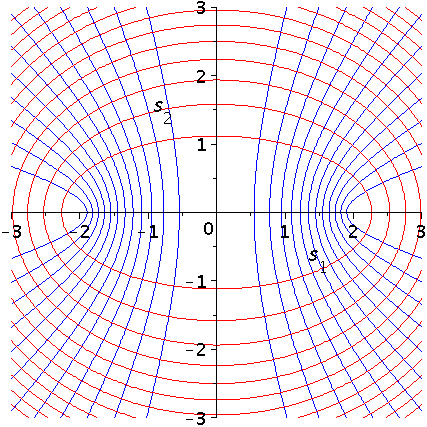}
    \label{fig:dual-frame-bounds-b}}
  \caption{The lower and upper frame bounds of dual frames $\Psi$ (to
    $\Phi$ defined in \eqref{eq:new-2x3-example}) as a function of
    $s_1$ and $s_2$. The two graphs in (a) meet at $s_1=\pm
    \sqrt{35}/3$ and $s_2=0$ which correspond to tight dual frames.}
  \label{fig:dual-frame-bounds}
\end{figure}

When the difference between the singular values of $\Psi$ goes to zero, the ellipses degenerate to a line segment (or even to a point if $\sigma_1=\sigma_2$). 
The limiting case corresponds to tight dual frames so Theorem~\ref{thm:tight-dual}(ii) applies, and we are forced to set $s_2=0$ to achieve row orthogonality of $M_\Psi$.  We then need to pick
$s_1$ such that the two row norms of $M_\Psi$ are equal, thus
\[
\abs{s_1}=\sqrt{\frac{1}{\sigma_2^2}-\frac{1}{\sigma_1^2}}
=\sqrt{\frac1{A^\Phi} - \frac{1}{B^\Phi}}=s_0,
\]
which shows that the above lower bound for the semi-major axis is sharp. 
\end{example}

\section*{Acknowledgments}
The authors would like to thank Christian Henriksen for valuable discussions. 
G.~Kutyniok acknowledges support by the Einstein Foundation Berlin, by Deutsche Forschungsgemeinschaft
(DFG) Grant SPP-1324 KU 1446/13 and DFG Grant KU 1446/14, and by the DFG Research Center {\sc Matheon} ``Mathematics for key technologies'' in Berlin.

\end{document}